\documentclass[reqno,11pt]{article} 
\usepackage[left=2.5cm,right=2.5cm,top=3cm,bottom=3cm,a4paper]{geometry}
\usepackage{amsmath, amssymb}
\usepackage{amsthm, amscd} 
\usepackage{xcolor}
\newcommand{\Mod}[1]{\ (\textup{mod}\ #1)}
\theoremstyle{plain} 
\newtheorem{theorem}{\indent\sc Theorem}[section]
\newtheorem{lemma}[theorem]{\indent\sc Lemma}

\newtheorem{proposition}[theorem]{\indent\sc Proposition}

\theoremstyle{definition} 
\newtheorem{definition}[theorem]{\indent\sc Definition}

\newtheorem{remark}[theorem]{\indent\sc Remark}
\newtheorem{example}[theorem]{\indent\sc Example}

\makeatletter
\def\address#1#2{\begingroup
\noindent\parbox[t]{7.8cm}{%
\small{\scshape\ignorespaces#1}\par\vskip1ex
\noindent\small{\itshape E-mail address}%
\/: #2\par\vskip4ex}\hfill%
\endgroup}%
\makeatother
\title{Binary quadratic forms and ray class groups} 
\author{
\textsc{Ick Sun Eum, Ja Kyung Koo and Dong Hwa Shin} 
}
\date{} 
\begin{document}

\allowdisplaybreaks

\maketitle

 \footnote{
\thanks{
The first author was supported by the Dongguk University Research Fund of 2016 and the National Research Foundation of Korea (NRF) grant funded by the Korea government (MSIT) (NRF-2017R1C1B5017567).
The third (corresponding) author was supported by the
Hankuk University of Foreign Studies Research Fund of 2018 and
the National Research Foundation of Korea (NRF) grant funded by the Korea
government (MSIT) (NRF-2017R1A2B1006578).} }

\begin{abstract}
Let $K$ be an imaginary quadratic field different from $\mathbb{Q}(\sqrt{-1})$ and $\mathbb{Q}(\sqrt{-3})$. For
a positive integer $N$, let $K_\mathfrak{n}$ be the ray class field of $K$ modulo $\mathfrak{n}=N\mathcal{O}_K$.
By using the congruence subgroup $\pm\Gamma_1(N)$, we construct an extended form class group whose operation is basically the Dirichlet composition, and explicitly show that this group
is isomorphic to the Galois group $\mathrm{Gal}(K_\mathfrak{n}/K)$. We also present algorithms to find all form classes and show how to multiply two form classes.
As an application,
we describe $\mathrm{Gal}(K_\mathfrak{n}^\mathrm{ab}/K)$ in terms of these
extended form class groups for which
$K_\mathfrak{n}^\mathrm{ab}$ is the maximal abelian extension of $K$
unramified outside prime ideals dividing $\mathfrak{n}$.
\end{abstract}

\maketitle

\section {Introduction}

Let $K$ be an imaginary quadratic field of discriminant $d_K$ with ring of integers $\mathcal{O}_K$.
Let $\mathcal{Q}(d_K)$ be the set
of primitive positive definite binary quadratic forms
$Q(x,\,y)=ax^2+bxy+cy^2$ ($\in\mathbb{Z}[x,\,y]$)
of discriminant $b^2-4ac=d_K$.
Define an equivalence relation on $\mathcal{Q}(d_K)$, called
the \textit{proper equivalence}, by
\begin{equation*}
Q'\sim Q\quad\Longleftrightarrow\quad
Q'\left(\begin{bmatrix}x\\y\end{bmatrix}\right)=
Q\left(\sigma\begin{bmatrix}x\\y\end{bmatrix}\right)
~\textrm{for some}~\sigma\in\mathrm{SL}_2(\mathbb{Z}).
\end{equation*}
Then, the set $\mathrm{C}(d_K)=\mathcal{Q}(d_K)/\sim$ of equivalence classes under Dirichlet composition
becomes a group, called the \textit{form class group}
of discriminant $d_K$ (\cite[Theorem 3.9]{Cox}).
\par
Let $I_K$ be the group
of fractional ideals of $K$ and $P_K$ be
its subgroup of principal fractional ideals.
It is a classical fact that the form class group $\mathrm{C}(d_K)$
is isomorphic to the ideal class group $\mathrm{Cl}(\mathcal{O}_K)=I_K/P_K$ as follows:
For each $Q\in\mathcal{Q}(d_K)$, let
$\omega_Q$ be the zero of $Q(x,\,1)$ in
the complex upper half-plane
$\mathbb{H}$.

\begin{theorem}\label{modulo1}
We have an isomorphism of groups
\begin{eqnarray*}
\phi~:~
\mathrm{C}(d_K)&\rightarrow&\mathrm{Cl}(\mathcal{O}_K)\\
\textrm{form class containing}~Q=
ax^2+bxy+cy^2
&\mapsto&\textrm{ideal class containing}~a[\omega_Q,\,1].
\end{eqnarray*}
\end{theorem}
\begin{proof}
See \cite[Theorem 7.7]{Cox}.
\end{proof}

\begin{remark}\label{aRemark}
In Theorem \ref{modulo1}, one can replace the integral ideal $a[\omega_Q,\,1]=[(-b+\sqrt{d_K})/2,\,a]$ by
the fractional ideal $[\omega_Q,\,1]$.
\end{remark}

On the other hand, let $H_K$ be the Hilbert class field of $K$
whose Galois group is isomorphic
to $\mathrm{Cl}(\mathcal{O}_K)$ (\cite[Theorem 8.10]{Cox} or \cite[Theorem 9.9 in Chapter V]{Janusz}).
The following theorem is a consequence of the theory of complex multiplication and Theorem \ref{modulo1}.

\begin{theorem}\label{Hilbert}
We have an isomorphism of groups
\begin{eqnarray*}
\mathrm{C}(d_K)&\rightarrow&\mathrm{Gal}(H_K/K)\\
\textrm{form class containing $Q$}&\mapsto&\left(j(\tau_K)\mapsto j(\omega_Q)\right),
\end{eqnarray*}
where $j(\tau)$ is the elliptic modular function
and $\tau_K$ is an element of $\mathbb{H}$
such that $\mathcal{O}_K=[\tau_K,\,1]$.
\end{theorem}
\begin{proof}
See \cite{Deuring}, \cite{Hasse} or \cite[Theorem 1 in Chapter 10]{Lang87}.
\end{proof}

Now, for a finite abelian extension $L$ of $K$
such that $L\supseteq H_K$,
it is natural to ask whether there
is some form class group that is isomorphic to $\mathrm{Gal}(L/K)$.
Since $\mathrm{Gal}(H_K/K)$ ($\simeq\mathrm{C}(d_K)$) is a quotient group of
$\mathrm{Gal}(L/K)$, if we loosen the
proper equivalence on $\mathrm{C}(d_K)$ induced from $\mathrm{SL}_2(\mathbb{Z})$, then we would expect to get a certain
new form class group isomorphic to $\mathrm{Gal}(L/K)$.
Here we note that $L$ is contained in some ray class field $K_\mathfrak{n}$
modulo $\mathfrak{n}=N\mathcal{O}_K$ for a positive integer $N$ (\cite[p. 149]{Cox}).

\begin{proposition}\label{CM}
Let $\mathcal{F}_N$
be the field of meromorphic modular functions of level $N$
whose Fourier coefficients lie in the $N$th cyclotomic field. Then
we have
\begin{equation*}
K_\mathfrak{n}=K(h(\tau_K)~|~h(\tau)\in\mathcal{F}_N~\textrm{is finite at}~\tau_K).
\end{equation*}
\end{proposition}
\begin{proof}
See \cite[Corollary to Theorem 2 in Chapter 10]{Lang87}.
\end{proof}

In this paper, we shall first construct a newly extended form class group
$\mathrm{C}_N(d_K)$ isomorphic to the ray class group $\mathrm{Cl}(\mathfrak{n})$,
through the equivalence relation induced from $\pm\Gamma_1(N)$ (Theorem \ref{group}).
It turns out that the binary operation on $\mathrm{C}_N(d_K)$
is essentially the Dirichlet composition on $\mathrm{C}(d_K)$
(Remark \ref{algorithm} (iv)).
\par
In view of Theorem \ref{Hilbert} and Proposition \ref{CM}
we shall further establish an isomorphism
\begin{eqnarray*}
\mathrm{C}_N(d_K)&\rightarrow&\mathrm{Gal}(K_\mathfrak{n}/K)\\
\begin{array}{c}\textrm{form class containing}\\
Q=ax^2+bxy+cy^2
\end{array}
&\mapsto&\left(h(\tau_K)\mapsto h_{\left[
\begin{smallmatrix}a&(b-b_K)/2\\0&1\end{smallmatrix}\right]}
(\omega_Q)~|~
h(\tau)\in\mathcal{F}_N~\textrm{is finite at}~\tau_K
\right),
\end{eqnarray*}
where $\min(\tau_K,\,\mathbb{Q})=x^2+b_Kx+c_K\in\mathbb{Z}[x]$ (Theorem \ref{CGisomorphism}).
This indicates that a form class $[ax^2+bxy+cy^2]$ in $\mathrm{C}_N(d_K)$
has perfect information on
an element
of $\mathrm{Gal}(K_\mathfrak{n}/K)$.
Of course, we shall present an algorithm in order to list all
representatives of
form classes in $\mathrm{C}_N(d_K)$ (Proposition \ref{list}) and give some examples.
\par
Let $K_\mathfrak{n}^\mathrm{ab}$
be the maximal abelian extension of $K$
unramified outside prime ideals dividing $\mathfrak{n}$.
As an application,
we shall construct a dense subset of
$\mathrm{Gal}(K_\mathfrak{n}^\mathrm{ab}/K)$,
equipped with Krull topology,
in terms of extended form class groups
(Theorem \ref{closure}).

\section {Extended form class groups as ray class groups}

Throughout this paper,
let $K$ be an imaginary quadratic field of discriminant $d_K$ other
than $\mathbb{Q}(\sqrt{-1})$ and $\mathbb{Q}(\sqrt{-3})$.
For a positive integer $N$, let
$\mathcal{Q}_N(d_K)$ be the set of primitive positive definite binary quadratic forms $Q(x,\,y)=ax^2+bxy+cy^2$ of discriminant $d_K$ such that $\gcd(N,\,a)=1$, that is,
\begin{equation*}
\mathcal{Q}_N(d_K)=\{ax^2+bxy+cy^2\in\mathcal{Q}(d_K)~|~
\gcd(N,\,a)=1\}.
\end{equation*}
By $\pm\Gamma_1(N)$ we mean the congruence subgroup of $\mathrm{SL}_2(\mathbb{Z})$ given by
\begin{equation*}
\pm\Gamma_1(N)=\left\{\sigma\in\mathrm{SL}_2(\mathbb{Z})~|~
\sigma\equiv\pm\begin{bmatrix}1&\mathrm{*}\\
0&1\end{bmatrix}\Mod{N}\right\}.
\end{equation*}

\begin{proposition}
The group $\pm\Gamma_1(N)$ acts on the set $\mathcal{Q}_N(d_K)$ on the right by
\begin{equation*}
Q^\sigma=Q\left(\sigma\begin{bmatrix}x\\y\end{bmatrix}\right)
\quad(\sigma\in\pm\Gamma_1(N),~Q\in\mathcal{Q}_N(d_K)).
\end{equation*}
\end{proposition}
\begin{proof}
Since $\mathrm{SL}_2(\mathbb{Z})$
acts on $\mathcal{Q}(d_K)$,
it suffices to show that $\pm\Gamma_1(N)$ preserves the set $\mathcal{Q}_N(d_K)$.
Let $Q(x,\,y)=ax^2+bxy+cy^2\in\mathcal{Q}_N(d_K)$ and
$\sigma\in\pm\Gamma_1(N)$. We then see that
\begin{eqnarray*}
Q\left(\sigma\begin{bmatrix}x\\y\end{bmatrix}\right)
&\equiv& Q\left(\pm\begin{bmatrix}1&s\\0&1\end{bmatrix}
\begin{bmatrix}x\\y\end{bmatrix}\right)
\Mod{N\mathbb{Z}[x,\,y]}\quad\textrm{for some}~s\in\mathbb{Z}\\
&\equiv&ax^2+
(2as+b)xy+(as^2+bs+c)y^2\Mod{N\mathbb{Z}[x,\,y]}.
\end{eqnarray*}
This claims that $Q\left(\sigma\begin{bmatrix}x\\y\end{bmatrix}\right)$ belongs to $\mathcal{Q}_N(d_K)$, as desired.
\end{proof}

\begin{definition}
Define an equivalence relation $\sim_N$ on the set $\mathcal{Q}_N(d_K)$ by
\begin{equation*}
Q\sim_N Q'\quad\Longleftrightarrow\quad
Q'\left(\begin{bmatrix}x\\y\end{bmatrix}\right)
=Q\left(\sigma\begin{bmatrix}x\\y\end{bmatrix}\right)
~\textrm{for some}~\sigma\in\pm\Gamma_1(N).
\end{equation*}
Denote by $\mathrm{C}_N(d_K)$ the set of equivalence classes,
namely,
\begin{equation*}
\mathrm{C}_N(d_K)=\mathcal{Q}_N(d_K)/\sim_N.
\end{equation*}
\end{definition}

Now, we are in need of the following basic lemma for later use.

\begin{lemma}\label{elementary}
Let $Q(x,\,y)=ax^2+bxy+cy^2\in\mathcal{Q}(d_K)$ and
$\sigma=\begin{bmatrix}r&s\\u&v\end{bmatrix}\in\mathrm{SL}_2(\mathbb{Z})$.
\begin{enumerate}
\item[\textup{(i)}] If $\omega\in\mathbb{H}$, then
\begin{equation*}
[\sigma(\omega),\,1]=\frac{1}{j(\sigma,\,\omega)}[\omega,\,1]
\quad\textrm{where}~j(\sigma,\,\omega)=u\omega+v.
\end{equation*}
\item[\textup{(ii)}] If $Q'\left(\begin{bmatrix}x\\y\end{bmatrix}\right)=
Q\left(\sigma\begin{bmatrix}x\\y\end{bmatrix}\right)$, then
\begin{equation*}
\omega_Q=\sigma(\omega_{Q'}).
\end{equation*}
\item[\textup{(iii)}] We have
\begin{equation*}
\mathrm{N}_{K/\mathbb{Q}}([\omega_Q,\,1])=\frac{1}{a},
\end{equation*}
where $\mathrm{N}_{K/\mathbb{Q}}(\cdot)$ is applied to
fractional ideals of $K$.
\end{enumerate}
\end{lemma}
\begin{proof}
\begin{enumerate}
\item[(i)] It follows from the fact $\sigma\in\mathrm{SL}_2(\mathbb{Z})$ that
\begin{equation*}
[\sigma(\omega),\,1]=
\left[\frac{r\omega+s}{u\omega+v},\,1\right]=
\frac{1}{u\omega+v}[r\omega+s,\,u\omega+v]
=\frac{1}{j(\sigma,\,\omega)}[\omega,\,1].
\end{equation*}
\item[(ii)]
\begin{equation*}
Q\left(\begin{bmatrix}\omega_Q\\1\end{bmatrix}\right)=0=
Q'\left(\begin{bmatrix}\omega_{Q'}\\1\end{bmatrix}\right)
=Q\left(\sigma\begin{bmatrix}\omega_{Q'}\\1\end{bmatrix}\right)
=j(\sigma,\,\omega_{Q'})^2Q\left(\begin{bmatrix}\sigma(\omega_{Q'})\\1\end{bmatrix}\right).
\end{equation*}
Since $\omega_Q,\,\omega_{Q'}\in\mathbb{H}$, we conclude $\omega_Q=\sigma(\omega_{Q'})$.
\item[(iii)]
\begin{equation*}
\mathrm{disc}_{K/\mathbb{Q}}([\omega_Q,\,1])=\left|\begin{matrix}
(-b+\sqrt{d_K})/2a & 1\\
(-b-\sqrt{d_K})/2a & 1
\end{matrix}\right|^2=\frac{d_K}{a^2}.
\end{equation*}
On the other hand, since
\begin{equation*}
\mathrm{disc}_{K/\mathbb{Q}}([\omega_Q,\,1])=\mathrm{N}_{K/\mathbb{Q}}([\omega_Q,\,1])^2d_K
\end{equation*}
(\cite[Proposition 13 in Chapter III]{Lang94}),
we achieve
\begin{equation*}
\mathrm{N}_{K/\mathbb{Q}}([\omega_Q,\,1])=\frac{1}{a}.
\end{equation*}

\end{enumerate}
\end{proof}

From now on, we let $\mathfrak{n}=N\mathcal{O}_K$ and
\begin{equation*}
\mathrm{Cl}(\mathfrak{n})=I_K(\mathfrak{n})/P_{K,\,1}(\mathfrak{n})
\end{equation*}
be the ray class group of $K$ modulo $\mathfrak{n}$, where
$I_K(\mathfrak{n})$ is the subgroup of
$I_K$ consisting of fractional ideals of $K$ prime to $\mathfrak{n}$ and $P_{K,\,1}(\mathfrak{n})$ is its subgroup
consisting of principal fractional ideals $\lambda\mathcal{O}_K$
with $\lambda\in K^*$
such that $\lambda\equiv^*1\Mod{\mathfrak{n}}$ (\cite[pp. 136--137]{Janusz}).

\begin{definition}\label{map}
Define a map
\begin{eqnarray*}
\phi_N~:~\mathrm{C}_N(d_K)&\rightarrow&\mathrm{Cl}(\mathfrak{n})\\
\mathrm{[}Q\mathrm{]}&\mapsto&\textrm{ray class containing}~[\omega_Q,\,1].
\end{eqnarray*}
Here, $[Q]$ stands for the form class containing $Q\in\mathcal{Q}_N(d_K)$.
\end{definition}

\begin{remark}
By Remark \ref{aRemark}, we see that $\phi_1=\phi$, the classical isomorphism
described in Theorem \ref{modulo1}.
\end{remark}

\begin{proposition}\label{well}
The map $\phi_N$ is well defined.
\end{proposition}
\begin{proof}
First, we shall show that
if $Q(x,\,y)=ax^2+bxy+cy^2\in\mathcal{Q}_N(d_K)$, then
the fractional ideal $[\omega_Q,\,1]$ is prime to $\mathfrak{n}$.
Observe that $a[\omega_Q,\,1]=[(-b+\sqrt{d_K})/2,\,a]$ is an integral ideal of $K$ with
\begin{equation*}
\mathrm{N}_{K/\mathbb{Q}}(a[\omega_Q,\,1])=a
\end{equation*}
by Lemma \ref{elementary} (iii).
This, together with the fact $\gcd(N,\,a)=1$, implies that
$[\omega_Q,\,1]$ is prime to $\mathfrak{n}$.
\par
Second, we shall show that if $Q,\,Q'\in\mathcal{Q}_N(d_K)$ such that $[Q]=[Q']$, then
$[\omega_Q,\,1]$ and $[\omega_{Q'},\,1]$ belong to the same ray class in $\mathrm{Cl}(\mathfrak{n})$.
Let
\begin{equation*}
Q'\left(\begin{bmatrix}x\\y\end{bmatrix}\right)=
a'x^2+b'xy+c'y^2=
Q\left(\sigma\begin{bmatrix}x\\y\end{bmatrix}\right)\quad
\textrm{for some}~\sigma=\begin{bmatrix}\mathrm{*}&\mathrm{*}\\
u&v\end{bmatrix}\in\pm\Gamma_1(N).
\end{equation*}
We then derive by Lemma \ref{elementary} (i) and (ii) that
\begin{equation*}
[\omega_{Q'},\,1]=(u\omega_{Q'}+v)[\sigma(\omega_{Q'}),\,1]
=(u\omega_{Q'}+v)[\omega_Q,\,1].
\end{equation*}
Since $\sigma\equiv\pm\begin{bmatrix}1&\mathrm{*}\\0&1\end{bmatrix}
\Mod{N}$ and $\gcd(N,\,a')=1$, we attain
\begin{equation*}
u\omega_{Q'}+v\equiv^* u\frac{-b'+\sqrt{d_K}}{2a'}+v
\equiv^*\pm1\Mod{\mathfrak{n}}.
\end{equation*}
This yields that $[\omega_Q,\,1]$ and $[\omega_{Q'},\,1]$
belong to the same ray class in $\mathrm{Cl}(\mathfrak{n})$.
\end{proof}

\begin{proposition}\label{injective}
The map $\phi_N$ is injective.
\end{proposition}
\begin{proof}
Suppose that
\begin{equation*}
\phi_N([Q])=\phi_N([Q'])\quad
\textrm{for some}~Q,\,Q'\in\mathcal{Q}_N(d_K),
\end{equation*}
and so
\begin{equation}\label{lambda}
[\omega_Q,\,1]=\lambda[\omega_{Q'},\,1]
\quad\textrm{for some}~\lambda\in K^*~\textrm{such that}~\lambda
\equiv^*1\Mod{\mathfrak{n}}.
\end{equation}
Then, we get by Theorem \ref{modulo1} that
\begin{equation*}
Q'\left(\begin{bmatrix}x\\y\end{bmatrix}\right)
=Q\left(\sigma\begin{bmatrix}x\\y\end{bmatrix}\right)
\quad\textrm{for some}~
\sigma=\begin{bmatrix}\mathrm{*}&\mathrm{*}\\
u&v\end{bmatrix}\in\mathrm{SL}_2(\mathbb{Z}).
\end{equation*}
And, it follows from Lemma \ref{elementary} (i), (ii) and (\ref{lambda}) that
\begin{equation*}
[\omega_{Q'},\,1]=j(\sigma,\,\omega_{Q'})[\sigma(\omega_{Q'}),\,1]=
(u\omega_{Q'}+v)[\omega_Q,\,1]=
\lambda(u\omega_{Q'}+v)[\omega_{Q'},\,1],
\end{equation*}
and hence
\begin{equation*}
\lambda(u\omega_{Q'}+v)\in\mathcal{O}_K^*=\{1,\,-1\}.
\end{equation*}
Now that $\lambda\equiv^*1\Mod{\mathfrak{n}}$, we deduce
\begin{equation}\label{pm1}
u\omega_{Q'}+v\equiv^*\pm1\Mod{\mathfrak{n}}.
\end{equation}
If we let $Q'(x,\,y)=a'x^2+b'xy+c'y^2$, then
we have $\mathcal{O}_K=[(-b'+\sqrt{d_K})/2,\,1]$ and
\begin{equation*}
u\omega_{Q'}+v\pm1=\frac{1}{a'}\left(u\frac{-b'+\sqrt{d_K}}{2}+a'(v\pm1)\right).
\end{equation*}
Thus, it follows from the fact $\gcd(N,\,a')=1$ and (\ref{pm1}) that
\begin{equation*}
u\equiv0,\quad v\equiv\pm1\Mod{N}.
\end{equation*}
Moreover, since $\det(\sigma)=1$, we obtain
\begin{equation*}
\sigma\equiv\pm\begin{bmatrix}1&\mathrm{*}\\
0&1\end{bmatrix}\Mod{N}.
\end{equation*}
Therefore, $Q$ and $Q'$ belong to the same class in $\mathrm{C}_N(d_K)$,
namely, $[Q]=[Q']$.
This proves the proposition.
\end{proof}

\begin{proposition}\label{surjective}
The map $\phi_N$ is surjective.
\end{proposition}
\begin{proof}
Let $C\in\mathrm{Cl}(\mathfrak{n})$.
Take an integral ideal $\mathfrak{a}$ in $C^{-1}$, and let
$\xi_1,\,\xi_2\in\ K^*$ such that
\begin{equation*}
\mathfrak{a}^{-1}=[\xi_1,\,\xi_2]\quad\textrm{and}\quad\xi=\frac{\xi_1}{\xi_2}\in\mathbb{H}.
\end{equation*}
Since $1\in\mathfrak{a}^{-1}$, one can write
\begin{equation}\label{1}
1=u\xi_1+v\xi_2\quad\textrm{for some}~u,\,v\in\mathbb{Z}.
\end{equation}
We then claim $\gcd(N,\,u,\,v)=1$. Otherwise, $d=\gcd(N,\,u,\,v)>1$, and so
$d\mathfrak{a}^{-1}=[d\xi_1,\,d\xi_2]$ contains $1$ by (\ref{1}), which implies
$d\mathcal{O}_K\supseteq\mathfrak{a}$. But, this
contradicts the fact that $\mathfrak{a}$ is prime to $\mathfrak{n}$.
Thus we may take a matrix
$\sigma=\begin{bmatrix}\mathrm{*} & \mathrm{*}\\
\widetilde{u}&\widetilde{v}\end{bmatrix}$ in $\mathrm{SL}_2(\mathbb{Z})$
such that
\begin{equation}\label{sigmaN}
\sigma\equiv\begin{bmatrix}\mathrm{*}&\mathrm{*}\\
u&v\end{bmatrix}\Mod{N}
\end{equation}
by the surjectivity of $\mathrm{SL}_2(\mathbb{Z})
\rightarrow\mathrm{SL}_2(\mathbb{Z}/N\mathbb{Z})$ (\cite[Lemma 1.38]{Shimura}).
If we set $\omega=\sigma(\xi)$, then we derive that
\begin{eqnarray*}
[\omega,\,1]&=&[\sigma(\omega),\,1]\\
&=&\frac{1}{\widetilde{u}\xi+\widetilde{v}}[\xi,\,1]
\quad\textrm{by Lemma \ref{elementary} (i)}\\
&=&\frac{\xi_2}{\widetilde{u}\xi_1+\widetilde{v}\xi_2}[\xi_1/\xi_2,\,1]
\quad\textrm{by the fact}~\xi=\xi_1/\xi_2\\
&=&\frac{1}{\widetilde{u}\xi_1+\widetilde{v}\xi_2}[\xi_1,\,\xi_2]\\
&=&\frac{1}{\widetilde{u}\xi_1+\widetilde{v}\xi_2}\mathfrak{a}^{-1}.
\end{eqnarray*}
Here, we note that
\begin{eqnarray*}
\widetilde{u}\xi_1+\widetilde{v}\xi_2-1&=&
\widetilde{u}\xi_1+\widetilde{v}\xi_2-(u\xi_1+v\xi_2)\quad
\textrm{by (\ref{1})}\\
&=&(\widetilde{u}-u)\xi_1+(\widetilde{v}-v)\xi_2\\
&\in&N\mathfrak{a}^{-1}\quad\textrm{by (\ref{sigmaN})},
\end{eqnarray*}
from which we see that
\begin{equation*}
\widetilde{u}\xi+\widetilde{v}\equiv^*1\Mod{\mathfrak{n}}.
\end{equation*}
Therefore, $[\omega,\,1]$ and $\mathfrak{a}^{-1}$ belong to the same ray class $C$. Thus, if we let $Q$ be the element of
$\mathcal{Q}_N(d_K)$ satisfying $\omega_Q=\omega$, then we conclude
\begin{equation*}
\phi_N([Q])=C.
\end{equation*}
\end{proof}

\begin{theorem}\label{group}
The set $\mathrm{C}_N(d_K)$ can be regarded as an abelian group isomorphic to the ray class group $\mathrm{Cl}(\mathfrak{n})$.
\end{theorem}
\begin{proof}
Define a binary operation $\cdot$ on $\mathrm{C}_N(d_K)$ by
\begin{equation*}
[Q]\cdot[Q']=\phi_N^{-1}(\phi_N([Q])\phi_N([Q'])),
\end{equation*}
where $\phi_N([Q])\phi_N([Q'])$ is the product of
ray classes in $\mathrm{Cl}(\mathfrak{n})$. This binary operation makes
$\mathrm{C}_N(d_K)$ an abelian group isomorphic to $\mathrm{Cl}(\mathfrak{n})$.
We shall describe the group operation on $\mathrm{C}_N(d_K)$ explicitly in
the following Remark \ref{algorithm} (iv).
\end{proof}

\begin{remark}\label{algorithm}
\begin{enumerate}
\item[(i)] If $M$ is a positive divisor of $N$ and $\mathfrak{m}=M\mathcal{O}_K$, then we have
by Definition \ref{map}
a commutative diagram of homomorphisms
\begin{equation*}
\begin{CD}
\mathrm{C}_N(d_K) @>\phi_N>> \mathrm{Cl}(\mathfrak{n})\\
@VVV @VVV\\
\mathrm{C}_M(d_K) @>\phi_M>> \mathrm{Cl}(\mathfrak{m})
\end{CD}
\end{equation*}
The vertical maps are natural projections.
\item[(ii)]
Let $\tau_K$ be the element of $\mathbb{H}$ induced by
the principal form
\begin{equation*}
\left\{\begin{array}{ll}
\displaystyle x^2+xy+\frac{1-d_K}{4}y^2 & \textrm{if}~d_K\equiv1\Mod{4},\\
\displaystyle x^2-\frac{d_K}{4}y^2 & \textrm{if}~d_K\equiv0\Mod{4}.
\end{array}\right.
\end{equation*}
Since $[\tau_K,\,1]=\mathcal{O}_K$, the principal form
gives rise to the identity element of $\mathrm{C}_N(d_K)$.
\item[(iii)] For a quadratic form $Q(x,\,y)=ax^2+bxy+cy^2\in\mathcal{Q}_N(d_K)$,
we want to find its inverse $[Q]^{-1}$ in $\mathrm{C}_N(d_K)$.
Let $\mathfrak{c}=a^{\varphi(N)}[\omega_Q,\,1]$,
where $\varphi$ is the Euler function. Then,
$\mathfrak{c}$ is an integral ideal of $K$ which belongs to the same
ray class as $[\omega_Q,\,1]$ because $a^{\varphi(N)}\equiv1\Mod{N}$.
Since $\mathfrak{c}\overline{\mathfrak{c}}=\mathrm{N}_{K/\mathbb{Q}}(\mathfrak{c})\mathcal{O}_K$
and $\mathrm{N}_{K/\mathbb{Q}}([\omega_Q,\,1])=1/a$ by Lemma
\ref{elementary} (iii), we get
\begin{equation*}
\mathfrak{c}^{-1}=\frac{1}{\mathrm{N}_{K/\mathbb{Q}}(\mathfrak{c})}
\overline{\mathfrak{c}}=\frac{1}{a^{\varphi(N)-1}}[-\overline{\omega}_Q,\,1];
\end{equation*}
and hence we obtain
\begin{equation*}
1=\frac{1}{a^{\varphi(N)-1}}(0\cdot(-\overline{\omega}_Q)+a^{\varphi(N)-1}\cdot1).
\end{equation*}
Take an element $\sigma$ in $\mathrm{SL}_2(\mathbb{Z})$ such that
\begin{equation*}
\sigma\equiv\begin{bmatrix}\mathrm{*}&\mathrm{*}\\
0&a^{\varphi(N)-1}\end{bmatrix}\Mod{N}.
\end{equation*}
Now, if we let $Q'\in\mathcal{Q}_N(d_K)$ satisfying
$\omega_{Q'}=\sigma(-\overline{\omega}_Q)$, then
we achieve by the proof of Proposition \ref{surjective} that $Q'$ and $\mathfrak{c}^{-1}$ give the same ray class.
Therefore, $[Q']$ is the inverse of $[Q]$ in $\mathrm{C}_N(d_K)$.
\item[(iv)] Let $Q_1(x,\,y)=a_1x^2+b_1xy+c_1y^2,\,Q_2(x,\,y)=
a_2x^2+b_2xy+c_2y^2\in\mathcal{Q}_N(d_K)$. We will describe
an algorithm how to find $[Q_1]\cdot[Q_2]$ explicitly.
One may take a matrix $\rho$ in $\mathrm{SL}_2(\mathbb{Z})$ so that
$Q_3(x,\,y)=a_3x^2+b_3xy+c_3y^2$ defined by
\begin{equation}\label{Q''}
Q_3\left(\begin{bmatrix}x\\y\end{bmatrix}\right)
=Q_2\left(\rho\begin{bmatrix}x\\y\end{bmatrix}\right)
\end{equation}
satisfies $\gcd(a_1,\,a_3,\,(b_1+b_3)/2)=1$ (\cite[Lemmas 2.3 and 2.25]{Cox}).
We then attain
\begin{equation}\label{B}
[\omega_{Q_1},\,1][\omega_{Q_3},\,1]=\left[\frac{-B+\sqrt{d_K}}{2a_1a_3},\,1\right],
\end{equation}
where $B$ is an integer for which
\begin{equation*}
B\equiv b_1\Mod{2a_1},\quad B\equiv b_3\Mod{2a_3}\quad
\textrm{and}\quad B^2\equiv d_K\Mod{4a_1a_3}
\end{equation*}
(\cite[Lemma 3.2 and (7.13)]{Cox}).
(This ideal multiplication gives us the Dirichlet composition on $\mathrm{C}_1(d_K)=\mathrm{C}(d_K)$.)
On the other hand, we know by Definition \ref{map} that
$\phi_N([Q_1])\phi_N([Q_2])$ is the ray class containing the fractional ideal
\begin{equation*}
\mathfrak{c}=[\omega_{Q_1},\,1][\omega_{Q_2},\,1].
\end{equation*}
Thus, we get that
\begin{eqnarray*}
\mathfrak{c}&=&[\omega_{Q_1},\,1][\rho(\omega_{Q_3}),\,1]
\quad\textrm{by (\ref{Q''}) and Lemma \ref{elementary} (ii)}\\
&=&\frac{1}{j(\rho,\,\omega_{Q3})}[\omega_{Q_1},\,1][\omega_{Q_3},\,1]
\quad\textrm{by Lemma \ref{elementary} (i)}\\
&=&\frac{1}{j(\rho,\,\omega_{Q_3})}
\left[\frac{-B+\sqrt{d_K}}{2a_1a_3},\,1\right]\quad\textrm{by (\ref{B})}.
\end{eqnarray*}
By the fact $\mathfrak{c}\overline{\mathfrak{c}}=
\mathrm{N}_{K/\mathbb{Q}}(\mathfrak{c})\mathcal{O}_K$ and Lemma \ref{elementary} (iii) we see that
\begin{equation*}
\mathfrak{a}=\mathfrak{c}^{-1}
=\frac{1}{\mathrm{N}_{K/\mathbb{Q}}(\mathfrak{c})}\overline{\mathfrak{c}}=
a_1a_2\overline{\mathfrak{c}}=[-a_1\overline{\omega}_{Q_1},\,a_1]
[-a_2\overline{\omega}_{Q_2},\,a_2]
\end{equation*}
is an integral ideal in the ray class $\left(\phi_N([Q_1])\phi([Q_2])\right)^{-1}$.
Now, by using the argument in the proof of Proposition \ref{surjective} one can have
$Q_4\in\mathcal{Q}_N(d_K)$ so that
$\phi_N([Q_4])$ is the ray class containing $\mathfrak{a}^{-1}=\mathfrak{c}$.
Therefore, we achieve
\begin{equation*}
[Q_4]=[Q_1]\cdot[Q_2].
\end{equation*}
\end{enumerate}
\end{remark}

\section {Extended form class groups as Galois groups}

Let $K_\mathfrak{n}$ be the ray class field of $K$ modulo $\mathfrak{n}$ ($=N\mathcal{O}_K$), that is, $K_\mathfrak{n}$ is the
unique abelian extension of $K$ whose Galois group
$\mathrm{Gal}(K_\mathfrak{n}/K)$ corresponds
to $\mathrm{Cl}(\mathfrak{n})$ via the Artin reciprocity map for $\mathfrak{n}$.
In this section, we shall establish
an isomorphism of $\mathrm{C}_N(d_K)$ onto $\mathrm{Gal}
(K_\mathfrak{n}/K)$ in a concrete way.
\par
Let $\mathcal{F}_N$ be the field of
meromorphic modular functions of level $N$ with
Fourier coefficients in $\mathbb{Q}(\zeta_N)$, where
$\zeta_N=e^{2\pi\mathrm{i}/N}$. It is well known that
$\mathcal{F}_N$ is a Galois extension of $\mathcal{F}_1$ with
\begin{equation*}
\mathrm{Gal}(\mathcal{F}_N/\mathcal{F}_1)\simeq
\mathrm{GL}_2(\mathbb{Z}/N\mathbb{Z})/\{\pm I_2\}
\end{equation*}
(\cite[Theorem 6.6]{Shimura}).
In particular, the subgroup $\mathrm{SL}_2(\mathbb{Z}/N\mathbb{Z})/\{\pm I_2\}$ of $\mathrm{GL}_2(\mathbb{Z}/N\mathbb{Z})/\{\pm I_2\}$
acts on the field $\mathcal{F}_N$ as follows: Let $h(\tau)\in\mathcal{F}_N$ and $\alpha\in\mathrm{SL}_2(\mathbb{Z}/N\mathbb{Z})/\{\pm I_2\}$.
Then we have
\begin{equation*}
h(\tau)^\alpha=h(\widetilde{\alpha}(\tau)),
\end{equation*}
where $\widetilde{\alpha}$ is any matrix in $\mathrm{SL}_2(\mathbb{Z})$ that
reduces to $\alpha$ via
$\mathrm{SL}_2(\mathbb{Z})\rightarrow
\mathrm{SL}_2(\mathbb{Z}/N\mathbb{Z})/\{\pm I_2\}$.

\begin{definition}
We call a Family
\begin{equation*}
\{h_\alpha(\tau)\}_{\alpha\in\mathrm{GL}_2(\mathbb{Z}/N\mathbb{Z})/\{\pm I_2\}}
\end{equation*}
of functions in $\mathcal{F}_N$ a \textit{Fricke family} of level $N$
if
\begin{equation*}
h_\alpha(\tau)^\beta=h_{\alpha\beta}(\tau)
\quad\textrm{for all}~\alpha,\,\beta\in\mathrm{GL}_2(\mathbb{Z}/N\mathbb{Z})/\{\pm I_2\}.
\end{equation*}
\end{definition}

\begin{remark}
For a Fricke family $\{h_\alpha(\tau)\}_\alpha$,
let $h(\tau)=h_{I_2}(\tau)$. Then we get
\begin{equation*}
h(\tau)^\alpha=h_{I_2}(\tau)^\alpha=
h_{I_2\alpha}(\tau)=h_\alpha(\tau)
\quad(\alpha\in\mathrm{GL}_2(\mathbb{Z}/N\mathbb{Z})/\{\pm I_2\}).
\end{equation*}
This shows that $\{h_\alpha(\tau)\}_\alpha$ is a family of Galois
conjugates of $h(\tau)=h_{I_2}(\tau)$
under $\mathrm{Gal}(\mathcal{F}_N/\mathcal{F}_1)$.
\end{remark}

For a class $C\in\mathrm{Cl}(\mathfrak{n})$ take an integral ideal $\mathfrak{a}$ in $C^{-1}$, and let $\xi_1,\,\xi_2\in K^*$ such that
\begin{equation*}
\mathfrak{a}^{-1}=[\xi_1,\,\xi_2]\quad\textrm{and}\quad
\xi=\frac{\xi_1}{\xi_2}\in\mathbb{H}.
\end{equation*}
Let $\tau_K$ be the element of $\mathbb{H}$ stated
in Remark \ref{algorithm} (ii).
Since $\mathcal{O}_K\subseteq\mathfrak{a}^{-1}$
and $\xi\in\mathbb{H}$, one can write
\begin{equation}\label{rsuv}
\begin{bmatrix}\tau_K\\1\end{bmatrix}=
\begin{bmatrix}r&s\\u&v\end{bmatrix}
\begin{bmatrix}\xi_1\\\xi_2\end{bmatrix}
\quad\textrm{for some}~
A=\begin{bmatrix}r&s\\u&v\end{bmatrix}\in M_2^+(\mathbb{Z}).
\end{equation}
It then follows that
\begin{equation*}
\begin{bmatrix}\tau_K&
\overline{\tau}_K\\1&1\end{bmatrix}=
\begin{bmatrix}r&s\\u&v\end{bmatrix}
\begin{bmatrix}\xi_1&\overline{\xi}_1\\
\xi_2&\overline{\xi}_2\end{bmatrix}.
\end{equation*}
Taking determinant and squaring, we obtain
\begin{equation*}
d_K=\det(A)^2\mathrm{disc}_{K/\mathbb{Q}}(\mathfrak{a}^{-1})
=\det(A)^2\mathrm{N}_{K/\mathbb{Q}}(\mathfrak{a})^{-2}d_K
\end{equation*}
(\cite[Proposition 13 in Chapter III]{Lang94}).
Thus, we deduce
$\det(A)=\mathrm{N}_{K/\mathbb{Q}}(\mathfrak{a})$ which is prime to $N$.

\begin{definition}\label{invariant}
Let $\{h_\alpha(\tau)\}_\alpha$ be a Fricke family of level $N$,
and let $C\in\mathrm{Cl}(\mathfrak{n})$.
Following the above notations, we define
\begin{equation*}
h_\mathfrak{n}(C)=h_{A}(\xi).
\end{equation*}
Here, we regard $A$ as an element of $\mathrm{GL}_2(\mathbb{Z}/N\mathbb{Z})/\{\pm I_2\}$.
\end{definition}

\begin{proposition}
The value $h_\mathfrak{n}(C)$ depends only on the ray class $C$, not
on the choices of $\mathfrak{a}$ and $\xi_1,\,\xi_2$.
\end{proposition}
\begin{proof}
First, let $\mathfrak{a}'$ be another integral ideal in $C^{-1}$. Then we have
\begin{equation*}
\mathfrak{a}'=\lambda\mathfrak{a}\quad
\textrm{for some}~\lambda\in K^*~\textrm{such that}~\lambda\equiv^*1\Mod{\mathfrak{n}},
\end{equation*}
and so
\begin{equation*}
\mathfrak{a}'^{-1}=\lambda^{-1}\mathfrak{a}^{-1}=
[\lambda^{-1}\xi_1,\,\lambda^{-1}\xi_2]\quad
\textrm{and}\quad
\frac{\lambda^{-1}\xi_1}{\lambda^{-1}\xi_2}=\frac{\xi_1}{\xi_2}=\xi\in\mathbb{H}.
\end{equation*}
We see from the fact
$\mathfrak{a},\,\mathfrak{a}'=\lambda\mathfrak{a}\subseteq\mathcal{O}_K$
that
\begin{equation*}
(\lambda-1)\mathfrak{a}\subseteq\mathcal{O}_K.
\end{equation*}
Moreover, since $\lambda\equiv^*1\Mod{\mathfrak{n}}$
and $\mathfrak{a}$ is prime to $\mathfrak{n}$, we attain
\begin{equation*}
(\lambda-1)\mathfrak{a}\subseteq \mathfrak{n},
\end{equation*}
and hence
\begin{equation*}
(\lambda-1)\mathcal{O}_K\subseteq N\mathfrak{a}^{-1}.
\end{equation*}
Thus we obtain by the fact $\mathcal{O}_K=[\tau_K,\,1]$ that
\begin{equation}\label{lambda-1}
(\lambda-1)\tau_K=N(a\xi_1+b\xi_2)\quad
\textrm{and}\quad
\lambda-1=N(c\xi_1+d\xi_2)\quad
\textrm{for some}~a,\,b,\,c,\,d\in\mathbb{Z}.
\end{equation}
On the other hand, since $\lambda\mathcal{O}_K\subseteq\lambda\mathfrak{a}'^{-1}=\mathfrak{a}^{-1}
=[\xi_1,\,\xi_2]$, we may write
\begin{equation}\label{r's'u'v'}
\begin{bmatrix}\lambda\tau_K\\\lambda\end{bmatrix}
=\begin{bmatrix}r'&s'\\u'&v'\end{bmatrix}\begin{bmatrix}
\xi_1\\\xi_2\end{bmatrix}
\quad\textrm{for some}~\begin{bmatrix}r'&s'\\u'&v'\end{bmatrix}\in M_2^+(\mathbb{Z}).
\end{equation}
One can then derive by (\ref{rsuv}), (\ref{lambda-1})
and (\ref{r's'u'v'}) that
\begin{equation*}
N\begin{bmatrix}a&b\\c&d\end{bmatrix}\begin{bmatrix}\xi_1\\\xi_2\end{bmatrix}
=\begin{bmatrix}r'&s'\\u'&v'\end{bmatrix}\begin{bmatrix}\xi_1\\\xi_2\end{bmatrix}
-\begin{bmatrix}r&s\\u&v\end{bmatrix}\begin{bmatrix}\xi_1\\\xi_2\end{bmatrix},
\end{equation*}
which yields
\begin{equation*}
\begin{bmatrix}
r'&s'\\u'&v'
\end{bmatrix}
\equiv\begin{bmatrix}r&s\\u&v\end{bmatrix}\Mod{N}.
\end{equation*}
\par
Second, let $\xi_1',\,\xi_2'\in\mathbb{H}$ such that
\begin{equation*}
\mathfrak{a}^{-1}=[\xi_1,\,\xi_2]=[\xi_1',\,\xi_2']\quad\textrm{and}\quad
\xi'=\frac{\xi_1'}{\xi_2'}\in\mathbb{H}.
\end{equation*}
We then express
\begin{equation*}
\begin{bmatrix}
\tau_K\\1
\end{bmatrix}=A'\begin{bmatrix}\xi_1'\\\xi_2'\end{bmatrix}
~\textrm{and}~
\begin{bmatrix}\xi_1'\\\xi_2'\end{bmatrix}=
B\begin{bmatrix}\xi_1\\\xi_2\end{bmatrix}
~\textrm{for some}~A'\in M_2^+(\mathbb{Z})~
\textrm{and}~B\in\mathrm{SL}_2(\mathbb{Z}),
\end{equation*}
and so by (\ref{rsuv}) we deduce
\begin{equation*}
A'\begin{bmatrix}\xi_1'\\\xi_2'\end{bmatrix}=
A\begin{bmatrix}\xi_1\\\xi_2\end{bmatrix}=
AB^{-1}\begin{bmatrix}\xi_1'\\\xi_2'\end{bmatrix}.
\end{equation*}
Hence we achieve
\begin{equation*}
\xi'=B(\xi)\quad\textrm{and}\quad A'=AB^{-1}.
\end{equation*}
Therefore we get that
\begin{equation*}
h_{A'}(\xi')=h_{AB^{-1}}(B(\xi))
=h_{AB^{-1}}(\tau)^B|_{\tau=\xi}
=h_{AB^{-1}B}(\tau)|_{\tau=\xi}
=h_{A}(\xi),
\end{equation*}
which proves the proposition.
\end{proof}

\begin{remark}
\begin{enumerate}
\item[(i)] If $C_0$ denotes the identity class in $\mathrm{Cl}(\mathfrak{n})$,
namely, $C_0$ is the ray class containing $\mathcal{O}_K=[\tau_K,\,1]$, then
\begin{equation*}
h_\mathfrak{n}(C_0)=h_{I_2}(\tau_K).
\end{equation*}
\item[(ii)] The invariant $h_\mathfrak{n}(C)$
is an analogue of the Siegel-Ramachandra invariant given in \cite[p. 235]{K-L} and \cite{Ramachandra}.
\end{enumerate}
\end{remark}

Let
\begin{equation*}
\mathrm{GL}_2(\mathbb{A}_\mathrm{f})=
\prod_{p\,:\,\textrm{primes}}{\phantom{1}}
\hspace{-0.6cm}'~~\mathrm{GL}_2(\mathbb{Q}_p),
\end{equation*}
where $'$ denotes the restricted product,
that is, for almost all $p$ the $p$-component
of an element of $\mathrm{GL}_2(\mathbb{A}_\mathrm{f})$ lies
in $\mathrm{GL}_2(\mathbb{Z}_p)$. One may express $\mathrm{GL}_2(\mathbb{A}_\mathrm{f})$ as
\begin{equation*}
\mathrm{GL}_2(\mathbb{A}_\mathrm{f})
=U\mathrm{GL}_2^+(\mathbb{Q})=\mathrm{GL}_2^+(\mathbb{Q})U,
\quad\textrm{where}~
U=\prod_{p\,:\,\textrm{primes}}\mathrm{GL}_2(\mathbb{Z}_p)
\end{equation*}
(\cite[Theorem 1 in Chapter 7]{Lang87}).
Let
\begin{equation*}
\mathcal{F}=\bigcup_{M=1}^\infty\mathcal{F}_M.
\end{equation*}
Then, we have a surjective homomorphism
\begin{equation*}
\sigma_\mathrm{f}:\mathrm{GL}_2(\mathbb{A}_\mathrm{f})\rightarrow\mathrm{Aut}(\mathcal{F})
\end{equation*}
with $\mathrm{Ker}(\sigma_\mathrm{f})=\mathbb{Q}^*$
(\cite[Theorems 4 and 6 in Chapter 7]{Lang87} or \cite[Theorem 6.23]{Shimura}). More precisely,
let $h(\tau)\in\mathcal{F}_N$ and
$\gamma=\alpha\beta\in\mathrm{GL}_2(\mathbb{A}_\mathrm{f})$
with $\alpha=(\alpha_p)_p\in U$ and $\beta\in\mathrm{GL}_2^+(\mathbb{Q})$.
By using the Chinese remainder theorem, one can find a unique matrix $\widetilde{\alpha}$ in $\mathrm{GL}_2(\mathbb{Z}/N\mathbb{Z})$ satisfying
$\widetilde{\alpha}\equiv\alpha_p\Mod{N}$ for all primes $p$ such that $p\,|\,N$.
We then obtain
\begin{equation}\label{Aut(F)}
h(\tau)^{\sigma_\mathrm{f}(\gamma)}=h^{\widetilde{\alpha}}(\beta(\tau))
\end{equation}
(\cite[Theorem 2 in Chapter 7 and p. 79]{Lang87}).
\par
For $\omega\in K\cap\mathbb{H}$, we have an embedding
\begin{equation*}
q_\omega:K^*\rightarrow\mathrm{GL}_2^+(\mathbb{Q})
\end{equation*}
defined by
\begin{equation*}
\xi\begin{bmatrix}\omega\\1\end{bmatrix}=
q_\omega(\xi)\begin{bmatrix}\omega\\1\end{bmatrix}
\quad(\xi\in K^*).
\end{equation*}
By continuity one can extend $q_\omega$ to an embedding
\begin{equation*}
q_{\omega,\,p}:K_p^*\rightarrow\mathrm{GL}_2(\mathbb{Q}_p)
\end{equation*}
for each prime $p$, and hence
to an embedding of idele groups
\begin{equation*}
q_\omega:\mathbb{A}_K^*\rightarrow\mathrm{GL}_2(\mathbb{A}_\mathrm{f})
\end{equation*}
(\cite[p. 149]{Lang87}).
Let $K^\mathrm{ab}$ be the maximal abelian extension of $K$.

\begin{proposition}[Shimura's reciprocity law]\label{Shimura}
Let $s$ be a idele of $K$ and
$(s^{-1},\,K)$ be the Artin symbol for $s^{-1}$ on $K^\mathrm{ab}$.
Let $\omega\in K\cap\mathbb{H}$ and  $h(\tau)\in\mathcal{F}$ which is finite at $\omega$. Then,
$h(\omega)$ lies in $K^\mathrm{ab}$ and satisfies
\begin{equation*}
h(\omega)^{(s^{-1},\,K)}=h(\tau)^{\sigma_\mathrm{f}(q_\omega(s))}|_{\tau=\omega}.
\end{equation*}
\end{proposition}
\begin{proof}
See \cite[Theorem 1 in Chapter 11]{Lang87} or \cite[Theorem 6.31 (i)]{Shimura}.
\end{proof}

\begin{theorem}\label{Artin}
Let $\{h_\alpha(\tau)\}_\alpha$ be a Fricke family of level $N$, and
let $C\in\mathrm{Cl}(\mathfrak{n})$.
If $h_\mathfrak{n}(C)$ is finite, then it belongs to $K_\mathfrak{n}$ and satisfies
\begin{equation*}
h_\mathfrak{n}(C)^{\sigma_\mathfrak{n}(C'^{-1})}=
h_\mathfrak{n}(CC')\quad\textrm{for all}~C'\in\mathrm{Cl}(\mathfrak{n})
\end{equation*}
where $\sigma_\mathfrak{n}:\mathrm{Cl}(\mathfrak{n})\rightarrow\mathrm{Gal}(K_\mathfrak{n}/K)$
is the Artin reciprocity map for $\mathfrak{n}$.
\end{theorem}
\begin{proof}
Let $\mathfrak{a}$ and $\mathfrak{a}'$ be integral ideals in
$C^{-1}$ and $C'^{-1}$, respectively. Take
$\xi_1,\,\xi_2,\,\xi_1'',\,\xi_2''\in K^*$ so that
\begin{equation*}
\mathfrak{a}^{-1}=[\xi_1,\,\xi_2]\quad\textrm{with}~
\xi=\frac{\xi_1}{\xi_2}\in\mathbb{H},
\end{equation*}
and
\begin{equation*}
(\mathfrak{a}\mathfrak{a}')^{-1}=[\xi_1'',\,\xi_2'']
\quad\textrm{with}~\xi''=\frac{\xi_1''}{\xi_2''}\in\mathbb{H}.
\end{equation*}
Since $\mathcal{O}_K\subseteq\mathfrak{a}^{-1}\subseteq(\mathfrak{a}\mathfrak{a}')^{-1}$,
we may write
\begin{equation}\label{tAx}
\begin{bmatrix}\tau_K\\1\end{bmatrix}=
A\begin{bmatrix}\xi_1\\\xi_2\end{bmatrix}
\quad\textrm{for some}~A\in M_2^+(\mathbb{Z})
\end{equation}
and
\begin{equation}\label{xBx'}
\begin{bmatrix}
\xi_1\\\xi_2
\end{bmatrix}
=B\begin{bmatrix}
\xi_1''\\\xi_2''
\end{bmatrix}\quad\textrm{for some}~B\in M_2^+(\mathbb{Z}).
\end{equation}
Let $s$ be an idele of $K$ such that for every prime $p$
\begin{equation}\label{sp}
\left\{\begin{array}{ll}
s_p=1 & \textrm{if}~p\,|\,N,\\
s_p\mathcal{O}_p=\mathfrak{a}'_p & \textrm{if}~p\nmid N.
\end{array}\right.
\end{equation}
Now that $\mathfrak{a}'$ is prime to $\mathfrak{n}$, we get
\begin{equation}\label{sc}
s_p^{-1}\mathcal{O}_p
=\mathfrak{a}'^{-1}_p\quad\textrm{for all primes}~p.
\end{equation}
Observe that for every prime $p$
\begin{equation*}
q_{\xi,\,p}(s_p^{-1})\begin{bmatrix}\xi_1\\\xi_2\end{bmatrix}
=\xi_2q_{\xi,\,p}(s_p^{-1})\begin{bmatrix}\xi\\1\end{bmatrix}
=\xi_2s_p^{-1}\begin{bmatrix}\xi\\1\end{bmatrix}=
s_p^{-1}\begin{bmatrix}\xi_1\\\xi_2\end{bmatrix}.
\end{equation*}
Thus,
\begin{equation*}
B^{-1}\begin{bmatrix}\xi_1\\
\xi_2\end{bmatrix}\quad\textrm{and}\quad
q_{\xi,\,p}(s_p^{-1})\begin{bmatrix}\xi_1\\
\xi_2\end{bmatrix}
\end{equation*}
are bases for the $\mathbb{Z}_p$-module $(\mathfrak{a}
\mathfrak{a}')_p^{-1}$
by (\ref{xBx'}) and (\ref{sc}). So, there exists $u_p\in\mathrm{GL}_2(\mathbb{Z}_p)$ such that
\begin{equation}\label{upB}
q_{\xi,\,p}(s_p^{-1})=u_pB^{-1}.
\end{equation}
If we let
\begin{equation*}
u=(u_p)_p\in\prod_{p\,:\,\textrm{primes}}\mathrm{GL}_2(\mathbb{Z}_p),
\end{equation*}
then we obtain
\begin{equation}\label{uB}
q_\xi(s^{-1})=uB^{-1}.
\end{equation}
Now, we derive that
\begin{eqnarray*}
h_\mathfrak{n}(C)^{(s,\,K)}&=&h_A(\xi)^{(s,\,K)}
\quad\textrm{by Definition \ref{invariant}}\\
&=&h_A(\tau)^{\sigma_\mathrm{f}(q_\xi(s^{-1}))}|_{\tau=\xi}\quad
\textrm{by Proposition \ref{Shimura}}\\
&=&h_A(\tau)^{\sigma_\mathrm{f}(uB^{-1})}|_{\tau=\xi}\quad\textrm{by (\ref{uB})}\\
&=&h_{Au}(B^{-1}(\tau))|_{\tau=\xi}\quad
\textrm{by (\ref{Aut(F)}),}\\
&&\hspace{3.2cm}\textrm{where $u$ is regarded as an element of $\mathrm{GL}_2(\mathbb{Z}/N\mathbb{Z})/\{\pm I_2\}$}\\
&=&h_{AB}(B^{-1}(\tau))|_{\tau=\xi}\quad\textrm{because
for every prime divisor $p$ of $N$}\\
&&\hspace{3.3cm}\textrm{we have $s_p=1$ by (\ref{sp}), and so
$u_pB^{-1}=I_2$ by (\ref{upB})}\\
&=&h_{AB}(B^{-1}(\xi))\\
&=&h_{AB}(\xi'')\quad\textrm{by (\ref{xBx'})}\\
&=&h_\mathfrak{n}(CC')
\quad\textrm{since}~
\begin{bmatrix}
\tau_K\\1
\end{bmatrix}=AB\begin{bmatrix}\xi_1''\\\xi_2''\end{bmatrix}
~\textrm{by (\ref{tAx}) and (\ref{xBx'})}.
\end{eqnarray*}
In particular, if $C'=C^{-1}$, then we see that
\begin{equation*}
h_\mathfrak{n}(C)=h_\mathfrak{n}(CC')^{(s^{-1},\,K)}=
h_\mathfrak{n}(C_0)^{(s^{-1},\,K)}
=h_{I_2}(\tau_K)^{(s^{-1},\,K)}.
\end{equation*}
This implies that $h_\mathfrak{n}(C)$ belongs to $K_\mathfrak{n}$
because $h_{I_2}(\tau_K)$ lies in $K_\mathfrak{n}$ by Proposition \ref{CM}.
Since
$\mathrm{ord}_\mathfrak{p}~s_p=\mathrm{ord}_\mathfrak{p}~\mathfrak{a}'$ for all primes
$p$ such that $p\nmid N$
and prime ideals lying above $p$
by (\ref{sp}), we achieve
\begin{equation*}
(s,\,K)|_{K_\mathfrak{n}}=\sigma_\mathfrak{n}(C'^{-1}).
\end{equation*}
Therefore, we conclude
\begin{equation*}
h_\mathfrak{n}(C)^{\sigma_\mathfrak{n}(C'^{-1})}=
h_\mathfrak{n}(CC').
\end{equation*}
\end{proof}

Let $\min(\tau_K,\,\mathbb{Q})=x^2+b_Kx+c_K\in\mathbb{Z}[x]$, and so
\begin{equation*}
\tau_K=\frac{-b_K+\sqrt{d_K}}{2}.
\end{equation*}

\begin{theorem}\label{CGisomorphism}
We have an isomorphism of groups
\begin{eqnarray*}
\mathrm{C}_N(d_K)&\rightarrow&\mathrm{Gal}(K_\mathfrak{n}/K)\\
\mathrm{[}ax^2+bxy+cy^2\mathrm{]}&\mapsto&
\left(
h(\tau_K)\mapsto h_{\left[
\begin{smallmatrix}a&(b-b_K)/2\\0&1\end{smallmatrix}\right]}
(\tfrac{-b+\sqrt{d_K}}{2a})~|~h(\tau)\in\mathcal{F}_N~
\textrm{is finite at $\tau_K$}\right).
\end{eqnarray*}
\end{theorem}
\begin{proof}
Let $Q(x,\,y)=ax^2+bxy+cy^2\in\mathcal{Q}_N(d_K)$. Then,
$C=\phi_N([Q])$ is the ray class containing
the fractional ideal
$\mathfrak{c}=[\omega_Q,\,1]$.
Since
\begin{equation*}
\mathfrak{c}^{-1}=\frac{1}{\mathrm{N}_{K/\mathbb{Q}}(\mathfrak{c})}
\overline{\mathfrak{c}}=\frac{1}{a}[-\overline{\omega}_Q,\,1]
\end{equation*}
by Lemma \ref{elementary} (iii),
$\mathfrak{a}=a^{\varphi(N)}\mathfrak{c}^{-1}$
is an integral ideal in $C^{-1}$.
It then follows that
\begin{equation*}
\mathfrak{a}^{-1}=\frac{1}{a^{\varphi(N)}}\mathfrak{c}=
\frac{1}{a^{\varphi(N)}}[\omega_Q,\,1]
\end{equation*}
and
\begin{equation*}
\begin{bmatrix}\tau_K\\1\end{bmatrix}=
\begin{bmatrix}
a^{\varphi(N)+1} & a^{\varphi(N)}(b-b_K)/2\\
0&a^{\varphi(N)}
\end{bmatrix}
\begin{bmatrix}
\omega_Q/a^{\varphi(N)}\\1/a^{\varphi(N)}
\end{bmatrix}.
\end{equation*}
Since $a^{\varphi(N)}\equiv1\Mod{N}$, we have
\begin{equation*}
h_\mathfrak{n}(C)=h_{\left[\begin{smallmatrix}
a & (b-b_K)/2\\
0&1
\end{smallmatrix}\right]}(\omega_Q).
\end{equation*}
Now, by composing two isomorphisms
\begin{eqnarray*}
\mathrm{C}_N(d_K)&\rightarrow&\mathrm{Cl}(\mathfrak{n})\\
\mathrm{[}ax^2+bxy+cy^2\mathrm{]}
&\mapsto&\textrm{ray class containing}~\mathrm{[}
(-b+\sqrt{d_K})/2a,\,1\mathrm{]}
\end{eqnarray*}
given in Theorem \ref{group} and
\begin{eqnarray*}
\mathrm{Cl}(\mathfrak{n})&\rightarrow&\mathrm{Gal}(K_\mathfrak{n}/K)\\
C&\mapsto&
\left(
h(\tau_K)=h_\mathfrak{n}(C_0)\mapsto
h_\mathfrak{n}(C_0)^{\sigma_\mathfrak{n}(C^{-1})}
=h_\mathfrak{n}(C)
~|~h(\tau)\in\mathcal{F}_N~\textrm{is finite at}~\tau_K\right)
\end{eqnarray*}
obtained by Theorem \ref{Artin}, we establish the theorem.
\end{proof}

\section {Explicit construction of extended form class groups}

In this section, we shall explain how to find representatives of forms classes in $\mathrm{C}_N(d_K)$.

\begin{lemma}\label{prime}
Let $Q(x,\,y)=ax^2+bxy+cy^2\in\mathcal{Q}_N(d_K)$ and
$u,\,v\in\mathbb{Z}$. Then, the fractional ideal
$(u\omega_Q+v)\mathcal{O}_K$ is prime to $\mathfrak{n}$
if and only if $Q(v,\,-u)$ is prime to $N$.
\end{lemma}
\begin{proof}
We deduce from the facts $\gcd(N,\,a)=1$ and $\mathfrak{n}=N\mathcal{O}_K$ that
\begin{eqnarray*}
&&(u\omega_Q+v)\mathcal{O}_K~\textrm{is prime to}~\mathfrak{n}\\
&\Longleftrightarrow&\textrm{the integral ideal}~
a(u\omega_Q+v)\mathcal{O}_K~\textrm{is prime to}~\mathfrak{n}\\
&\Longleftrightarrow&
\mathrm{N}_{K/\mathbb{Q}}(a(u\omega_Q+v))~\textrm{is prime to}~N.
\end{eqnarray*}
Hence, we obtain that
\begin{eqnarray*}
\mathrm{N}_{K/\mathbb{Q}}(a(u\omega_Q+v))&=&
a^2(u\omega_Q+v)(u\overline{\omega}_Q+v)\\
&=&a^2(u^2\omega_Q\overline{\omega}_Q+uv(\omega_Q+\overline{\omega}_Q)+v^2)\\
&=&a^2(u^2(c/a)+uv(-b/a)+v^2)\\
&=&a(cu^2-buv+av^2)\\
&=&aQ(v,\,-u).
\end{eqnarray*}
This proves the lemma.
\end{proof}

Let $P_K(\mathfrak{n})$ be the subgroup of $I_K(\mathfrak{n})$
consisting of principal fractional ideals prime to $\mathfrak{n}$.

\begin{lemma}\label{Cuv}
Let $Q(x,\,y)=ax^2+bxy+cy^2\in\mathcal{Q}_N(d_K)$ and
$C\in P_K(\mathfrak{n})/P_{K,\,1}(\mathfrak{n})\subseteq\mathrm{Cl}(\mathfrak{n})$.
Then we have
\begin{equation*}
C=[(u\omega_Q+v)\mathcal{O}_K]\quad
\textrm{for some $u,\,v\in\mathbb{Z}$
such that}~\gcd(N,\,Q(v,\,-u))=1.
\end{equation*}
\end{lemma}
\begin{proof}
Take an integral ideal $\mathfrak{c}$ in $C$. Since
$\mathcal{O}_K=[a\omega_Q,\,1]$, we get
\begin{equation*}
\mathfrak{c}=(ta\omega_Q+v)\mathcal{O}_K\quad
\textrm{for some}~t,\,v\in\mathbb{Z}.
\end{equation*}
Set $u=ta$. Then, the lemma follows from Lemma \ref{prime}.
\end{proof}

Define an equivalence relation $\equiv_N$ on $\mathbb{Z}^2$ by
\begin{equation*}
\begin{bmatrix}r\\s\end{bmatrix}
\equiv_N\begin{bmatrix}u\\v\end{bmatrix}
\quad\Longleftrightarrow\quad
\begin{bmatrix}r\\s\end{bmatrix}\equiv\pm\begin{bmatrix}u\\v\end{bmatrix}
\Mod{N}.
\end{equation*}

\begin{lemma}\label{samerayclass}
Let $Q(x,\,y)=ax^2+bxy+cy^2\in\mathcal{Q}_N(d_K)$,
and let
$\begin{bmatrix}r\\s\end{bmatrix},\,
\begin{bmatrix}u\\v\end{bmatrix}\in\mathbb{Z}^2$
such that $\gcd(N,\,Q(s,\,-r))=\gcd(N,\,Q(v,\,-u))=1$.
Then, $(r\omega_Q+s)\mathcal{O}_K$ and $(u\omega_Q+v)\mathcal{O}_K$
represent the same ray class in $\mathrm{Cl}(\mathfrak{n})$ if and only if
\begin{equation*}
\begin{bmatrix}r\\s\end{bmatrix}\equiv_N\begin{bmatrix}u\\v\end{bmatrix}.
\end{equation*}
\end{lemma}
\begin{proof}
By Lemma \ref{prime}, both
$(r\omega_Q+s)\mathcal{O}_K$ and
$(u\omega_Q+v)\mathcal{O}_K$ are prime to $\mathfrak{n}$.
Then we see that
\begin{eqnarray*}
&&\textrm{$(r\omega_Q+s)\mathcal{O}_K$
and $(u\omega_Q+v)\mathcal{O}_K$ represent the same ray class in $\mathrm{Cl}(\mathfrak{n})$}\\
&\Longleftrightarrow&
\left(\frac{r\omega_Q+s}{u\omega_Q+v}\right)\mathcal{O}_K\in P_{K,\,1}(\mathfrak{n})\\
&\Longleftrightarrow&
\frac{r\omega_Q+s}{u\omega_Q+v}\equiv^*\pm1\Mod{\mathfrak{n}}\quad
\textrm{because}~\mathcal{O}_K^*=\{1,\,-1\}\\
&\Longleftrightarrow&
a(r\omega_Q+s)\equiv^*\pm a(u\omega_Q+v)\Mod{\mathfrak{n}}\\
&\Longleftrightarrow&
(r\pm u)(a\omega_Q)+(s\pm v)a\in\mathfrak{n}
\quad\textrm{since}~a\omega_Q\in\mathcal{O}_K\\
&\Longleftrightarrow&
r\pm u\equiv(s\pm v)a\equiv0\Mod{N}\quad\textrm{due to}~
\mathfrak{n}=N\mathcal{O}_K=[Na\omega_Q,\,N]\\
&\Longleftrightarrow&
\begin{bmatrix}r\\s\end{bmatrix}
\equiv\pm\begin{bmatrix}u\\v\end{bmatrix}\Mod{N}
\quad\textrm{by the fact}~\gcd(N,\,a)=1\\
&\Longleftrightarrow&
\begin{bmatrix}r\\s\end{bmatrix}
\equiv_N\begin{bmatrix}u\\v\end{bmatrix}.
\end{eqnarray*}
\end{proof}

\begin{proposition}\label{list}
One can list elements of $\mathrm{C}_N(d_K)$ explicitly.
\end{proposition}
\begin{proof}
Note first that
\begin{equation}\label{PG}
\mathrm{C}(d_K)\simeq\mathrm{Gal}(K_\mathfrak{n}/K)/
\mathrm{Gal}(K_\mathfrak{n}/H_K)
\quad\textrm{and}\quad
P_K(\mathfrak{n})/P_{K,\,1}(\mathfrak{n})\simeq\mathrm{Gal}(K_\mathfrak{n}/H_K).
\end{equation}
One can readily find reduced forms $Q_1,\,Q_2,\,\ldots,\,Q_n$ in
$\mathcal{Q}(d_K)$ which represent all classes in $\mathrm{C}(d_K)$
(\cite[Theorem 2.8]{Cox}). Take $\sigma_1,\,\sigma_2,\,\ldots,\,\sigma_n\in
\mathrm{SL}_2(\mathbb{Z})$ for which
\begin{equation*}
Q_i'\left(\begin{bmatrix}x\\y\end{bmatrix}\right)
=Q_i\left(\sigma_i\begin{bmatrix}
x\\y\end{bmatrix}\right)\quad(i=1,\,2,\,\ldots,\,n)
\end{equation*}
belong to $\mathcal{Q}_N(d_K)$ (\cite[Lemmas 2.3 and 2.25]{Cox}).
Then,
\begin{equation*}
\left\{\left[[\omega_{Q_i'},\,1]\right]\in\mathrm{Cl}(\mathfrak{n})~|~
i=1,\,2,\,\ldots,\,n\right\}
\end{equation*}
is a subset of $\mathrm{Cl}(\mathfrak{n})$
whose image under $\mathrm{Cl}(\mathfrak{n})\rightarrow
\mathrm{Cl}(\mathcal{O}_K)$ is all of $\mathrm{Cl}(\mathcal{O}_K)$.
Furthermore, for each $i\in\{1,\,2,\,\ldots,\,n\}$, we attain
by Lemmas \ref{prime}, \ref{Cuv} and \ref{samerayclass} that
\begin{equation}\label{PP}
P_K(\mathfrak{n})/P_{K,\,1}(\mathfrak{n})=\left\{
\left[(u\omega_{Q_i'}+v)\mathcal{O}_K\right]~|~
\begin{bmatrix}u\\v\end{bmatrix}\in\mathbb{Z}^2/\equiv_N~\textrm{such
that}~\gcd(N,\,Q_i'(v,-u))=1\right\}.
\end{equation}
\par
Now, let $C\in\mathrm{Cl}(\mathfrak{n})$.
By (\ref{PG}) and (\ref{PP}), there
is one and only one pair
$i\in\{1,\,2,\,\ldots,\,n\}$ and
$\begin{bmatrix}u\\v\end{bmatrix}\in\mathbb{Z}^2/\equiv_N$
with $\gcd(N,\,Q_i'(v,\,-u))=1$ so that
\begin{equation*}
C=\left[\frac{1}{u\omega_{Q_i'}+v}[\omega_{Q_i'},\,1]\right].
\end{equation*}
Take a matrix $\rho$ in $\mathrm{SL}_2(\mathbb{Z})$ satisfying
\begin{equation*}
\rho\equiv\begin{bmatrix}\mathrm{*}&\mathrm{*}\\
u&v\end{bmatrix}\Mod{N}.
\end{equation*}
Since
\begin{equation*}
\frac{j(\rho,\,\omega_{Q_i'})}{u\omega_{Q_i'}+v}\equiv^*1\Mod{\mathfrak{n}},
\end{equation*}
we get by Lemma \ref{elementary} (i) that
\begin{equation*}
C=\left[
\frac{1}{j(\rho,\,\omega_{Q_i'})}[\omega_{Q_i'},\,1]\right]
=\left[[\rho(\omega_{Q_i'}),\,1]\right].
\end{equation*}
If $\widetilde{Q}\in\mathcal{Q}_N(d_K)$ such that
$\omega_{\widetilde{Q}}=\rho(\omega_{Q_i'})$, then
we attain
\begin{equation*}
C=\phi_N([\widetilde{Q}])=
\phi_N\left(\left[Q_i'\left(\rho^{-1}\begin{bmatrix}x\\y\end{bmatrix}\right)\right]\right).
\end{equation*}
This completes the proof.
\end{proof}

\begin{example}
Let $K=\mathbb{Q}(\sqrt{-2})$ and $N=3$. There
is only one reduced form
\begin{equation*}
Q_1=x^2+2y^2
\end{equation*}
of discriminant $d_K=-8$. Set $Q_1'=Q_1$.
By the proof of Proposition \ref{list} one can find two representatives $\widetilde{Q}_1$, $\widetilde{Q}_2$ of form classes in $\mathrm{C}_3(-8)$ whose zeros are
\begin{equation*}
\begin{bmatrix}1&0\\0&1\end{bmatrix}(\omega_{Q_1'})
\quad\textrm{and}\quad
\begin{bmatrix}0&-1\\1&0\end{bmatrix}(\omega_{Q_1'}).
\end{equation*}
And, we obtain
\begin{equation*}
\mathrm{C}_3(-8)=\{[x^2+2y^2],~[2x^2+y^2]\},
\end{equation*}
and hence $\mathrm{C}_3(-8)\simeq\mathbb{Z}/2\mathbb{Z}$.
\end{example}

\begin{example}\label{-20}
Let $K=\mathbb{Q}(\sqrt{-5})$ and $N=2$.
Then there are two reduced forms of discriminant $d_K=-20$, namely,
\begin{equation*}
Q_1=x^2+5y^2\quad\textrm{and}\quad
Q_2=2x^2+2xy+3y^2.
\end{equation*}
Let
\begin{equation*}
Q_1'=Q_1\quad\textrm{and}\quad
Q_2'=Q_2\left(\begin{bmatrix}0&-1\\1&0\end{bmatrix}
\begin{bmatrix}x\\y\end{bmatrix}
\right)=3x^2-2xy+2y^2
\end{equation*}
with
\begin{equation*}
\omega_{Q_1'}=\sqrt{-5}\quad\textrm{and}\quad
\omega_{Q_2'}=\frac{1+\sqrt{-5}}{3}.
\end{equation*}
By the proof of Proposition \ref{list} we have
four representatives
$\widetilde{Q}_1$, $\widetilde{Q}_2$, $\widetilde{Q}_3$, $\widetilde{Q}_4$ of form classes in $\mathrm{C}_2(-20)$
with zeros
\begin{equation*}
\begin{bmatrix}
1&0\\0&1
\end{bmatrix}(\omega_{Q_1'}),~
\begin{bmatrix}0&-1\\1&0\end{bmatrix}(\omega_{Q_1'}),~
\begin{bmatrix}1&0\\0&1\end{bmatrix}(\omega_{Q_2'}),~
\begin{bmatrix}1&0\\1&1\end{bmatrix}(\omega_{Q_2'}),\quad\textrm{respectively}.
\end{equation*}
Thus, we achieve
\begin{equation*}
\mathrm{C}_2(-20)=\{
[x^2+5y^2],~[5x^2+y^2],~[3x^2-2xy+2y^2],~[7x^2-6xy+2y^2]\}.
\end{equation*}
Note that
\begin{equation*}
\widetilde{Q}_4'=\widetilde{Q}_4\left(\begin{bmatrix}
1&0\\2&1\end{bmatrix}\begin{bmatrix}x\\y\end{bmatrix}\right)=
3x^2+2xy+2y^2\sim_2\widetilde{Q_4}.
\end{equation*}
We then see by using the argument in
Remark \ref{algorithm} (iii) that
\begin{equation*}
[\widetilde{Q}_4]^{-1}=[\widetilde{Q}_4']^{-1}=[\widetilde{Q_3}].
\end{equation*}
Therefore, we conclude
\begin{equation*}
\mathrm{C}_2(-20)\simeq\mathbb{Z}/4\mathbb{Z}.
\end{equation*}
\end{example}

\begin{example}
Let $K=\mathbb{Q}(\sqrt{-5})$ and $N=6$.
Let $Q_1$ and $Q_2$ be reduced forms of discriminant $d_K=-20$
stated in Example \ref{-20}.
In this case, we let
\begin{equation*}
Q_1'=Q_1\quad\textrm{and}\quad
Q_2'=Q_2\left(\begin{bmatrix}1&-1\\1&0\end{bmatrix}\begin{bmatrix}
x\\y\end{bmatrix}\right)=7x^2-6xy+2y^2
\end{equation*}
with
\begin{equation*}
\omega_{Q_1'}=\sqrt{-5}\quad
\textrm{and}\quad
\omega_{Q_2'}=\frac{3+\sqrt{-5}}{7}.
\end{equation*}
There are
eight representatives of form classes in $\mathrm{C}_6(-20)$
with zeros
\begin{eqnarray*}
&&\begin{bmatrix}
1&0\\0&1
\end{bmatrix}(\omega_{Q_1'}),~
\begin{bmatrix}
0&-1\\1&0
\end{bmatrix}(\omega_{Q_1'}),~
\begin{bmatrix}
1&1\\2&3
\end{bmatrix}(\omega_{Q_1'}),~
\begin{bmatrix}
-1&-1\\3&2
\end{bmatrix}(\omega_{Q_1'}),\\
&&\begin{bmatrix}
1&0\\0&1
\end{bmatrix}(\omega_{Q_2'}),~
\begin{bmatrix}
0&-1\\1&3
\end{bmatrix}(\omega_{Q_2'}),~
\begin{bmatrix}
1&1\\2&3
\end{bmatrix}(\omega_{Q_2'}),~
\begin{bmatrix}
1&0\\3&1
\end{bmatrix}(\omega_{Q_2'})
\end{eqnarray*}
by Proposition \ref{list}.
Hence we attain
\begin{eqnarray*}
&&\mathrm{C}_6(-20)=
\{[x^2+5y^2],~[5x^2+y^2],~
[29x^2-26xy+6y^2],~[49x^2+34xy+6y^2],\\
&&[7x^2-6xy+2y^2],~[83x^2+48xy+7y^2],~
[107x^2-80xy+15y^2],~[43x^2-18xy+2y^2]\}.
\end{eqnarray*}
\end{example}

We close this section by
mentioning that the equivalence relation on $\mathcal{Q}_N(d_K)$ induced from the congruence subgroup
\begin{equation*}
\Gamma_0(N)=\left\{\sigma\in\mathrm{SL}_2(\mathbb{Z})~|~
\sigma\equiv\begin{bmatrix}\mathrm{*}&\mathrm{*}\\
0&\mathrm{*}\end{bmatrix}\Mod{N}\right\}
\end{equation*}
gives us a new form class group
as a quotient group of $\mathrm{C}_N(d_K)$.

\begin{remark}
Let $\mathcal{O}=[N\tau_K,\,1]$ be an order of conductor $N$ in $K$
and $\mathrm{Cl}(\mathcal{O})$ be the ideal class group
of the order $\mathcal{O}$ (\cite[p. 123]{Cox}).
Then we know that
\begin{equation*}
\mathrm{Cl}(\mathcal{O})\simeq I_K(\mathfrak{n})/P_{K,\,\mathbb{Z}}(\mathfrak{n}),
\end{equation*}
where $P_{K,\,\mathbb{Z}}(\mathfrak{n})$ is the subgroup
of $I_K(\mathfrak{n})$ consisting of
principal fractional ideals $\lambda\mathcal{O}_K$
such that
\begin{equation*}
\lambda\equiv^*m\Mod{\mathfrak{n}}\quad
\textrm{for some}~m\in\mathbb{Z}~\textrm{with}~\gcd(N,\,m)=1
\end{equation*}
(\cite[Proposition 7.22]{Cox}).
And, by the existence theorem of class field theory, there
is a unique abelian extension $H_\mathcal{O}$ of $K$
for which
\begin{equation*}
\mathrm{Gal}(H_\mathcal{O}/K)\simeq
I_K(\mathfrak{n})/P_{K,\,\mathbb{Z}}(\mathfrak{n})
\end{equation*}
via the Artin reciprocity map for $\mathfrak{n}$.
Define an equivalence relation $\sim_\mathcal{O}$ on $\mathcal{Q}_N(d_K)$ by
\begin{equation*}
Q\sim_\mathcal{O}Q'\quad\Longleftrightarrow\quad
Q'\left(\begin{bmatrix}x\\y\end{bmatrix}\right)=
Q\left(\sigma\begin{bmatrix}x\\y\end{bmatrix}\right)
~\textrm{for some}~\sigma\in\Gamma_0(N).
\end{equation*}
Let
\begin{equation*}
\mathrm{C}_\mathcal{O}(d_K)=\mathcal{Q}_N(d_K)/\sim_\mathcal{O},
\end{equation*}
and define a map
\begin{eqnarray*}
\phi_\mathcal{O}~:~\mathrm{C}_\mathcal{O}(d_K)&\rightarrow&
I_K(\mathfrak{n})/P_{K,\,\mathbb{Z}}(\mathfrak{n})\\
\left[Q\right]&\mapsto&\textrm{class containing}~
\left[\omega_Q,\,1\right].
\end{eqnarray*}
Then, in a similar way as we did in Propositions \ref{well}, \ref{injective} and \ref{surjective} one can readily show that $\phi_\mathcal{O}$ is well defined and bijective.
Therefore, we may consider the
set $\mathrm{C}_\mathcal{O}(d_K)$
as a quotient group of $\mathrm{C}_N(d_K)$, which
is isomorphic to $\mathrm{Gal}(H_\mathcal{O}/K)$.
\end{remark}

\section {The maximal abelian extension unramified outside prime ideals dividing $\mathfrak{n}$}

Let $K_\mathfrak{n}^\mathrm{ab}$ be the maximal abelian extension of $K$
unramified outside prime ideals dividing $\mathfrak{n}$ ($=N\mathcal{O}_K$).
If $N=1$, then $K_\mathfrak{n}^\mathrm{ab}$ is nothing but the Hilbert class field $H_K$ of $K$.
So, we assume $N\geq2$.
As an application,
we shall describe $\mathrm{Gal}(K_\mathfrak{n}^\mathrm{ab}/K)$
in view of extended form class groups.
Here we shall regard $\mathrm{Gal}(K_\mathfrak{n}^\mathrm{ab}/K)$
as a topological group equipped with Krull topology:
For each $\rho\in\mathrm{Gal}(K_\mathfrak{n}^\mathrm{ab}/K)$, we take
the cosets
\begin{equation*}
\rho\mathrm{Gal}(K_\mathfrak{n}^\mathrm{ab}/F)
\end{equation*}
as a basis of open neighborhoods of $\rho$, where $F$ runs through
all finite (abelian) subextensions of $K_\mathfrak{n}^\mathrm{ab}/K$
(\cite[$\S$I.1]{Neukirch}).
\par
If $L$ is a finite abelian extension of $K$
unramified outside prime ideals dividing $\mathfrak{n}$, then its conductor also
divides $\mathfrak{n}^\ell$ for some $\ell\geq1$. Thus
$L$ is contained in the ray class field $K_{\mathfrak{n}^\ell}$ (\cite[p. 116]{Shimura}), and hence
we get
\begin{equation*}
K_\mathfrak{n}^\mathrm{ab}=\bigcup_{\ell\geq1}
K_{\mathfrak{n}^\ell}.
\end{equation*}
Furthermore,
now that
\begin{equation*}
K_\mathfrak{n}\subseteq K_{\mathfrak{n}^2}\subseteq\cdots
\subseteq K_{\mathfrak{n}^\ell}\subseteq\cdots,
\end{equation*}
we attain the isomorphisms
\begin{equation}\label{inverselimit}
\mathrm{Gal}(K_\mathfrak{n}^\mathrm{ab}/K)\simeq
\varprojlim_{\ell}\mathrm{Gal}(K_{\mathfrak{n}^\ell}/K)
\simeq\varprojlim_{\ell}\mathrm{C}_{N^\ell}(d_K)
\end{equation}
of topological groups by Theorem \ref{CGisomorphism} (\cite[$\S$2 in Appendix]{Washington}).
Here, the inverse system
$\{\mathrm{C}_{N^\ell}(d_K)\}_{\ell}$
is given by the natural surjections
$\mathrm{C}_{N^n}(d_K)\leftarrow\mathrm{C}_{N^m}(d_K)$ ($1\leq n\leq m$).
And we observe
\begin{equation*}
\mathcal{Q}_{N^\ell}(d_K)=\mathcal{Q}_N(d_K)\quad\textrm{for all}~\ell\geq1.
\end{equation*}
For each $Q\in\mathcal{Q}_N(d_K)$ and $\ell\geq1$, denote by
\begin{equation*}
[Q]_{N^\ell}=\textrm{the form class containing $Q$ in}~\mathrm{C}_{N^\ell}(d_K).
\end{equation*}
Then we have
\begin{equation*}
\varprojlim_\ell\mathrm{C}_{N^\ell}(d_K)=
\left\{([Q_1]_N,\,[Q_2]_{N^2},\,\ldots,\,[Q_\ell]_{N^\ell},\,\ldots)\in\prod_{\ell}\mathrm{C}_{N^\ell}(d_K)~\bigg|~
[Q_{\ell+1}]_{N^\ell}=[Q_\ell]_{N^\ell}~\textrm{for all}~\ell\geq1\right\}.
\end{equation*}
Now, define an equivalence relation $\sim_{N^\infty}$ on the set
$\mathcal{Q}_N(d_K)$ by
\begin{equation*}
Q\sim_{N^\infty}Q'\quad\Longleftrightarrow\quad
Q\sim_{N^\ell}Q'~\textrm{for all}~\ell\geq1.
\end{equation*}
For each $Q\in\mathcal{Q}_N(d_K)$, let $[Q]_{N^\infty}$
be the form class containing $Q$ in $\mathrm{Q}_N(d_K)/\sim_{N^\infty}$.
We also define a map
\begin{eqnarray*}
\iota~:~\mathcal{Q}_N(d_K)/\sim_{N^\infty}&\rightarrow&\varprojlim_{\ell}
\mathrm{C}_{N^\ell}(d_K)\\
\mathrm{[}Q\mathrm{]}_{N^\infty}&\mapsto&([Q]_N,\,[Q]_{N^2},\,\ldots,\,[Q]_{N^\ell},\,\ldots).
\end{eqnarray*}
Then it is straightforward that $\iota$ is well defined and injective.

\begin{lemma}\label{Cclosure}
We derive
\begin{equation*}
\varprojlim_\ell\mathrm{C}_{N^\ell}(d_K)=\overline{\iota(\mathcal{Q}_N(d_K)/\sim_{N^\infty})}.
\end{equation*}
\end{lemma}
\begin{proof}
Let $\displaystyle([Q_1]_N,\,[Q_2]_{N^2},\,\ldots,\,[Q_\ell]_{N^\ell},\,\ldots)
\in\varprojlim_{\ell}\mathrm{C}_{N^\ell}(d_K)$ be given.
For every $\ell\geq1$, we see that
\begin{eqnarray*}
\iota([Q_\ell]_{N^\infty})&=&([Q_\ell]_N,\,[Q_\ell]_{N^2},\,\ldots,\,[Q_\ell]_{N^\ell},\,
[Q_\ell]_{N^{\ell+1}},\,\ldots)\\
&=&([Q_1]_N,\,[Q_2]_{N^2},\,\ldots,\,[Q_\ell]_{N^\ell},\,
[Q_\ell]_{N^{\ell+1}},\,\ldots).
\end{eqnarray*}
Considering the Krull topology on $\mathrm{Gal}(K_\mathfrak{n}^\mathrm{ab}/K)$
we conclude that $\iota(\mathcal{Q}_N(d_K)/\sim_{N^\infty})$ is
a dense subset of $\displaystyle\varprojlim_{\ell}\mathrm{C}_{N^\ell}(d_K)$.
\end{proof}

For $T=\begin{bmatrix}1&1\\0&1\end{bmatrix}$, let us
define a new equivalence relation $\sim_T$
on $\mathcal{Q}_N(d_K)$ by
\begin{equation*}
Q\sim_T Q'\quad\Longleftrightarrow\quad
Q'\left(\begin{bmatrix}x\\y\end{bmatrix}\right)=
Q\left(\sigma\begin{bmatrix}x\\y\end{bmatrix}\right)~\textrm{for some}~
\sigma\in\langle -I_2,\,T\rangle.
\end{equation*}

\begin{lemma}\label{samerelation}
Two equivalence relations
$\sim_{N^\infty}$ and $\sim_T$
are the same.
\end{lemma}
\begin{proof}
Let $Q(x,\,y)=ax^2+bxy+cy^2$
and $Q'(x,\,y)=a'x^2+b'xy+c'y^2$ be two
elements of $\mathcal{Q}_N(d_K)$.
Since $\langle -I_2,\,T\rangle$ is contained in $\pm\Gamma_1(N^\ell)$
for all $\ell\geq1$, it is immediate that
if $Q\sim_T Q'$, then $Q\sim_{N^\infty} Q'$.
\par
Conversely, assume that $Q\sim_{N^\infty}Q'$. Then, for each $\ell\geq1$,
there is $\sigma_\ell\in\pm\Gamma_1(N^\ell)$ such that
\begin{equation*}
Q'\left(\begin{bmatrix}x\\y\end{bmatrix}\right)=
Q\left(\sigma_\ell\begin{bmatrix}x\\y\end{bmatrix}\right).
\end{equation*}
Hence it follows from
\begin{equation*}
Q\left(\sigma_1\begin{bmatrix}x\\y\end{bmatrix}\right)=
Q\left(\sigma_\ell\begin{bmatrix}x\\y\end{bmatrix}\right)
\end{equation*}
that
\begin{equation*}
Q\left(\sigma_1\sigma_\ell^{-1}\begin{bmatrix}x\\y\end{bmatrix}\right)
=Q\left(\begin{bmatrix}x\\y\end{bmatrix}\right),
\end{equation*}
which yields that $\sigma_1\sigma_\ell^{-1}$ belongs to
the stabilizer subgroup $\mathrm{Stab}(Q)$ ($\subseteq\mathrm{SL}_2(\mathbb{Z})$) of $Q$.
Since we are assuming $K\neq\mathbb{Q}(\sqrt{-1}),\,\mathbb{Q}(\sqrt{-3})$,
$\mathrm{Stab}(Q)=\{I_2,\,-I_2\}$; and hence
$\sigma_1=\sigma_\ell$ or $\sigma_1=-\sigma_\ell$.
Owing to the assumption $N\geq2$ we achieve
\begin{equation*}
\sigma_1\in\bigcap_{\ell\geq1}\pm\Gamma_1(N^\ell)=\langle -I_2,\,T\rangle.
\end{equation*}
Therefore, we conclude $Q\sim_T Q'$.
\end{proof}

\begin{lemma}\label{Trepresent}
Let $Q(x,\,y)=ax^2+bxy+cy^2$ and
$Q'(x,\,y)=a'x^2+b'xy+c'y^2$
be two elements of $\mathcal{Q}_N(d_K)$. Then,
\begin{equation*}
Q\sim_T Q'\quad\Longleftrightarrow\quad
a=a'~\textrm{and}~a~\textrm{divides}~\frac{b-b'}{2}.
\end{equation*}
\end{lemma}
\begin{proof}
Observe that $b$ and $b'$ have the same parity
by the discriminant condition
\begin{equation}\label{d_K}
b^2-4ac=b'^2-4a'c'=d_K.
\end{equation}
We then see that
\begin{eqnarray*}
Q\sim_T Q'&\Longleftrightarrow&Q'\left(\begin{bmatrix}x\\y\end{bmatrix}\right)
=Q\left(\begin{bmatrix}1&s\\0&1\end{bmatrix}\begin{bmatrix}x\\y\end{bmatrix}\right)~
\textrm{for some}~s\in\mathbb{Z}\\
&\Longleftrightarrow&
a'x^2+b'xy+c'y^2=ax^2+(2ax+b)xy+(a^2s+bs+c)y^2~
\textrm{for some}~s\in\mathbb{Z}\\
&\Longleftrightarrow&a'=a~\textrm{and}~b'=2as+b~\textrm{for some}~s\in\mathbb{Z}~
\textrm{by (\ref{d_K})}\\
&\Longleftrightarrow&a=a'~\textrm{and}~a~\textrm{divides}~(b-b')/2.
\end{eqnarray*}
\end{proof}

\begin{theorem}\label{closure}
The set $\mathcal{Q}_N(d_K)/\sim_T$
can be viewed as a dense subset of $\mathrm{Gal}(K_\mathfrak{n}^\mathrm{ab}/K)$.
\end{theorem}
\begin{proof}
Let
\begin{equation*}
\phi:\varprojlim_{\ell}\mathrm{C}_{N^\ell}(d_K)
\rightarrow\mathrm{Gal}(K_\mathfrak{n}^\mathrm{ab}/K)
\end{equation*}
be the isomorphism obtained in (\ref{inverselimit}). Then we get
by Lemmas \ref{Cclosure} and \ref{samerelation}
\begin{equation*}
\mathrm{Gal}(K_\mathfrak{n}^\mathrm{ab}/K)
=\overline{(\phi\circ\iota)(\mathcal{Q}_N(d_K)/\sim_T)}.
\end{equation*}
Moreover, Lemma \ref{Trepresent} enables us
to distinguish
different classes in $\mathcal{Q}_N(d_K)/\sim_T$
from one another.
\end{proof}

\bibliographystyle{amsplain}

\address{
Department of Mathematics Education\\
Dongguk University-Gyeongju\\
Gyeongju-si, Gyeongsangbuk-do 38066\\
Republic of Korea} {zandc@dongguk.ac.kr}
\address{
Department of Mathematical Sciences \\
KAIST \\
Daejeon 34141\\
Republic of Korea} {jkkoo@math.kaist.ac.kr}
\address{
Department of Mathematics\\
Hankuk University of Foreign Studies\\
Yongin-si, Gyeonggi-do 17035\\
Republic of Korea} {dhshin@hufs.ac.kr}

\end{document}